\documentclass[11pt]{amsart}
\usepackage{mathrsfs}
\usepackage{amsfonts}
\usepackage{amsthm}
\usepackage{amssymb}
\usepackage{latexsym}

\theoremstyle{plain}
\newtheorem{theorem}[equation]{Theorem}
\newtheorem{lemma}[equation]{Lemma}
\newtheorem{corollary}[equation]{Corollary}

\newtheorem{proposition}[equation]{Proposition}
\newtheorem{definition}[equation]{Definition}
\theoremstyle{remark}
\newtheorem{remark}[equation]{Remark}

\numberwithin{equation}{section}

\newcommand{\sjump}{\hskip .2 cm}

\newcommand{\ssubset}{\subset\subset}

\begin{document}

\title[Duality and Smoothing]{Duality of holomorphic function spaces and smoothing
properties of 
the Bergman projection}
\author{A.-K. Herbig, J. D. McNeal, \& E. J. Straube}
\thanks{Research supported in part by Austrian Science Fund FWF grants Y377 and V187N13
(Herbig), National Science Foundation grants DMS--0752826 (McNeal) and DMS--0758534
(Straube), and by the Erwin Schr\"{o}dinger International Institute for
Mathematical Physics.}

\address{Department of Mathematics, \newline University of Vienna, Vienna, Austria}
\email{anne-katrin.herbig@univie.ac.at}
\address{Department of Mathematics, \newline Ohio State University, Columbus, Ohio, USA}
\email{mcneal@math.ohio-state.edu}
\address{Department of Mathematics, \newline Texas A\&M University, College Station, Texas, USA}
\email{straube@math.tamu.edu}

\begin{abstract}
Let $\Omega\ssubset\mathbb{C}^{n}$ be a domain with smooth boundary,  whose
Bergman projection $B$ maps the Sobolev space $H^{k_{1}}(\Omega)$ (continuously) into
$H^{k_{2}}(\Omega)$. We establish two smoothing results: (i) the full Sobolev norm
$\|Bf\|_{k_{2}}$ is controlled by $L^2$ derivatives of $f$ taken along \textit{a single,
distinguished direction} (of order $\leq k_{1}$), and (ii)  the projection of a
\textit{conjugate holomorphic} function in $L^{2}(\Omega)$ is automatically in
$H^{k_{2}}(\Omega)$. There are obvious corollaries for when $B$ is globally regular.
\end{abstract}

\maketitle

\section{{Introduction}}\label{intro}

In this paper, we consider the range of the Bergman projection, $B$, associated to 
a smoothly bounded domain $\Omega\ssubset\mathbb{C}^n$ through the following
question: what functions defined on $\Omega$ does $B$ map to elements of
$C^{\infty}(\overline{\Omega})$? Our particular interest is in finding families of
functions ${\mathcal F}$, ${\mathcal
F}\not\subset C^{\infty}(\overline{\Omega})$, such that $B\left({\mathcal F}\right)\subset
C^{\infty}(\overline{\Omega})$. 

We present two partial answers to this general question, both under the hypothesis that
$B$ satisfies Condition R of Bell--Ligocka \cite{BellLigocka}, which says that
$B\left(C^{\infty}(\overline{\Omega})\right) \subseteq C^{\infty}(\overline{\Omega})$.
In the first, we show that an $L^2$ function $f$ which has square-integrable derivatives
of all orders only in a {\it single, distinguished direction} is
necessarily mapped by $B$ to a function in $C^{\infty}(\overline{\Omega})$. 
The direction is distinguished by being both tangential to $b\Omega$, the boundary of $\Omega$, and not contained in
the complex tangent space to $b\Omega$ (the last property is called \textit{complex transversal}, for short).
Note that no smoothness about $f$ is assumed except in this one direction, yet $Bf$ is smooth up to the
boundary in all directions. This `partial smoothing' property of $B$ has a different
character than traditionally studied mapping properties of $B$,
which concentrate on whether $B$ preserves various Banach space structures, and was
discovered recently by the first two authors \cite{HerMcN10}. Our first result here
generalizes the main theorem in \cite{HerMcN10}, which established partial smoothing of
$B$ under the hypothesis that the $\bar\partial$-Neumann operator on $\Omega$ is compact.
The second result in this paper says that all {\it conjugate holomorphic} functions in
$L^2(\Omega)$ are mapped to $C^{\infty}(\overline{\Omega})$ by $B$. This result differs
from the first as functions in $\overline{A^0}(\Omega)$, where ${A^0}(\Omega)$ denotes the
space of square-integrable holomorphic functions on $\Omega$, need have no derivatives in
$L^2(\Omega)$. 

We now state our results more precisely. $H^{k}(\Omega)$ denotes the standard
$L^{2}$ Sobolev space of order $k$, and $H_{T}^{k}(\Omega)$ denotes the Sobolev space of
order $k$ involving only derivatives in direction $T$, see Definition \ref{D:TSobolev}.

\begin{theorem}\label{T:main}
 Let $\Omega\ssubset\mathbb{C}^{n}$ be a smoothly bounded domain and let $T$
be a smooth vector field  on $\overline{\Omega}$ that is tangential and complex
transversal at the boundary. Suppose that there exist a pair $(k_{1}, k_{2}) \in
\mathbb{N} \times \mathbb{N}$ and a constant $C_{1}>0$ such that
\begin{align} \label{main1}
  \left\|Bf\right\|_{k_{2}}\leq
C_{1}\left\|f\right\|_{k_{1}}\qquad\sjump\forall\sjump f\in H^{k_{1}}(\Omega).
\end{align}
Then there exists a
constant $C_{2}>0$, such that
  \begin{align}\label{main2}
    \|Bf\|_{k_{2}}\leq C_{2}\|f\|_{k_{1},T}\qquad\sjump\forall\sjump f\in
H_{T}^{k_{1}}(\Omega) .
  \end{align}
\end{theorem}

Note that in \eqref{main1}, automatically $k_{1} \geq k_{2}$. We emphasize that
\eqref{main2} is a genuine estimate: if the
right hand side is finite, i.e., $f$ has $T$ derivatives up to order $k_{1}$ in
$L^{2}(\Omega)$, then $Bf \in H^{k_{2}}(\Omega)$, and the estimate holds.

Condition R is equivalent to the statement that for each $k_{2} \in \mathbb{N}$, there
exists $k_{1} \in \mathbb{N}$ such that $B: H^{k_{1}}(\Omega) \rightarrow
H^{k_{2}}(\Omega)$. Combined with the Sobolev Lemma to the effect that
$\cap_{k=0}^{\infty}H^{k}(\Omega) =
C^{\infty}(\overline{\Omega})$, Theorem \ref{T:main} implies in particular the
$C^{\infty}(\overline{\Omega})$ result described in the second paragraph above:

\begin{corollary}\label{C:main}
Assume $\Omega$ and $T$ are as in Theorem \ref{T:main}, and that Condition R holds for
$\Omega$. Then $B$ maps $H_{T}^{\infty}:=\cap_{k=0}^{\infty}H_{T}^{k}(\Omega)$ (continuously) to
$C^{\infty}(\overline{\Omega})$.
\end{corollary}

The method we use to prove Theorem \ref{T:main} is rather general and applicable in
other situations, i.e., to other operators connected to elliptic PDEs and to other spaces
of holomorphic functions. 

 The proof consists of two distinct steps. The
first step is to show that when $B$ satisfies the regularity condition \eqref{main1}, the
$H^{k_{2}}(\Omega)$ norm of a holomorphic function $g$ can be estimated by pairing $g$,
in the ordinary $L^2$ inner product, against holomorphic functions contained in the unit ball of
$H^{-k_{1}}(\Omega)$ (see Proposition \ref{P:pieceofduality} below). This is a
special fact about holomorphic functions (but it can also be established for solutions to
other homogeneous elliptic systems, see \cite{Bell82b, Straube84, Ligocka86} and Appendix
B in \cite{Boas87}). This special type of duality for holomorphic functions originates in
the work of Bell \cite{Bell81a, Bell82a}, and was subsequently extended and refined in
\cite{Bell82b, BellBoas84, Komatsu84, Straube84, Barrett95}. For our purposes, duality is
used to reduce the problem of estimating $\|Bf\|_{k_{2}}$ significantly. In order to
illustrate how this works, assume for simplicity that $k_{1}=k_{2}=1$. Then, for some constant $c>0$,
\begin{align}\label{E:Bf1norm}
\|Bf\|_1&\leq c \sup\left\{\left|(Bf,h)\right|: h\in A^{1}(\Omega),\sjump\|h\|_{-1}\leq 1
    \right\}\notag\\
    &=c\sup\left\{\left|(f,h)\right|: h\in A^{1}(\Omega),\sjump\|h\|_{-1}\leq 1
    \right\}.
\end{align}
The key point is that the self-adjointness of $B$ in the ordinary $L^2$ inner product,
which yields the last equality, `eliminates' $B$ from the right-hand side of
\eqref{E:Bf1norm}. This obviates the need to study the commutator of $B$ with differential
operators in order to bound the left-hand side of \eqref{E:Bf1norm}; estimating such
commutators is a difficult problem in general since $B$ is an abstractly given integral
operator. (This effect also occurs when one can use vector fields with holomorphic
coefficients to control Sobolev norms of $Bf$, see \cite{Barrett86}.) A similar use of
duality seems applicable to other self-adjoint operators, and also to $B$ itself in other
scales of Banach spaces besides $H^k(\Omega)$. 

Let $T$ be the vector field from Theorem \ref{T:main}. Once \eqref{E:Bf1norm} is in hand,
the second step (still assuming $k_{1}=k_{2}=1$) consists of replacing $h$, on the
right-hand side of \eqref{E:Bf1norm}, by
the sum of the $\overline{T}$ derivative of a function, ${\mathcal H}_1$, and of another
function, ${\mathcal H}_2$, both of whose $L^2$ norms are uniformly bounded (when
$\|h\|_{-1} \leq 1$). Since $T$ is tangential, it follows that
\begin{align*} 
(f,h)=(f,\overline{T}{\mathcal H}_1+{\mathcal H}_2) 
&=(\overline{T}^{*}f, {\mathcal H}_1)+(f,{\mathcal H}_2) \\
&\leq C\left( \|Tf\| +\|f\|\right),
\end{align*}
where the last inequality follows after
noticing that the $L^{2}$ adjoint $\overline{T}^{*}$ differs from $-T$
by a term  bounded in $L^{2}$. In this very simple way,
the full Sobolev norm of $Bf$ of order $1$ is controlled by the $L^2$ norm of the derivative of $f$
in the special direction $T$, provided ${\mathcal H}_1, {\mathcal H}_2$ exist. 

Proving the existence of ${\mathcal H}_1$ and ${\mathcal H}_2$, and similar functions for 
higher powers of $\overline{T}$, is conceptually simple, but, for higher powers,
somewhat technical. This fact accounts for much of the length of this paper.
To simplify matters temporarily suppose that $T$ is of the form
 $T = T_{1} + L$, where both $T_{1}$
and $L$ are tangential at the boundary, $T_{1}$ is real, and $L$ is of type $(1,0)$. Note
that $T_{1}$ is complex transversal because $T$ is; consequently, if $J$ denotes the
complex structure map on $\mathbb{C}^{n}$, $JT_{1}$ is transversal to the boundary of
$\Omega$. Let now $h\in A^{1}(\Omega)$ with $\|h\|_{-1}\leq 1$ be given. The
goal is to write $h=\overline{T}{\mathcal H}_1+ {\mathcal H}_2$ where $\|{\mathcal
H}_1\|_{L^2}, \|{\mathcal H}_2\|_{L^2}$ are bounded by constants depending only on
$\Omega$ and $T$. To this end, consider first
$\mathfrak{A}\circ(JT_{1}h)$, where $\mathfrak{A}$ is the operator of
`anti-differentiation along the direction $JT_{1}$'. Then
\begin{equation}\label{Intro:1}
h(z)= \mathfrak{A}\circ JT_{1}h (z) +h(q),
\end{equation}
where $q=q(z)$ varies in a fixed compact subset of $\Omega$. Because $h$ is holomorphic,
the contributions given by $h(q)$ are easily shown to be bounded in $L^2$ and are folded into the function $\mathcal{H}_2$. Furthermore, the Cauchy--Riemann
equations show that
$JT_{1}(h)= -i\,T_{1}(h)$. Thus 
\begin{align}\label{Intro:2}
\mathfrak{A}\circ JT_{1}(h)&= \mathfrak{A}\circ  (-i)\, T_{1}h\notag \\
&= -i\left(\mathfrak{A}\circ \overline{T}(h) - \mathfrak{A}\left(\overline{L}
h\right)\right) \; .
\end{align}
But $\overline{L}h$ vanishes since $h$ is holomorphic. As a last step, we commute
$\overline{T}$ with $\mathfrak{A}$ in \eqref{Intro:2}:
\begin{equation}\label{Intro:3}
\mathfrak{A}\circ \overline{T}(h)=\overline{T}\circ\mathfrak{A}(h)+\left[ \mathfrak{A}, 
\overline{T}\right](h).
\end{equation}
The commutator in \eqref{Intro:3} is straightforward to analyze since both $\mathfrak{A}$ 
and $\overline{T}$ are explicit operators; this term forms the final component of the
function $\mathcal{H}_2$. The term $-i\,\mathfrak{A}(h)$ is the function $\mathcal{H}_1$.
The needed inequality $\|\mathfrak{A}(h)\|\leq C \|h\|_{-1}$ is a consequence of
Hardy's inequality, which says that $\mathfrak{A}$ gains a factor 
of the boundary distance $d(z)$, together with the fact that such a factor gains a
derivative in the case of holomorphic functions:
\begin{align}\label{Hardy-hol}
\|\mathfrak{A}(h)\|\leq c_{1}\|d(z)h\| \leq c_{2}\|h\|_{-1}\; .
\end{align}

\bigskip

Coming to our second result, we start with an observation about the Bergman projection on
the unit disk. Modulo constants, conjugate holomorphic functions on the unit disc are
orthogonal to the Bergman space. Equivalently: their Bergman projections are constant
functions. Of course, this fails on general (even planar) domains. Our second result says
that nevertheless, if the Bergman projection satisfies a regularity estimate such as 
\eqref{main1}, projections of conjugate holomorphic functions are still as good as
projections of functions smooth up to the boundary: they belong to $H^{k_{2}}(\Omega)$. In
particular, if Condition R holds, these projections are themselves smooth up to the
boundary.

\begin{theorem}\label{T:holconjsmoothing}
Let $\Omega\ssubset\mathbb{C}^{n}$ be a smoothly bounded domain. 
Suppose that the Bergman projection $B$ on $\Omega$ satisfies the regularity
condition
\eqref{main1}. Then there is a constant $C > 0$ such that
\begin{align}\label{E:holconjsmoothingfinite}
\|Bf\|_{k_{2}} \leq C\|f\|
\end{align}
for all conjugate holomorphic functions $f$ in $L^{2}(\Omega)$. 
\end{theorem}

In contrast to Theorem \ref{T:main}, smoothing here takes place in all directions. But
just as with Theorem \ref{T:main}, there is an immediate corollary for when Condition R
holds.
\begin{corollary}\label{holconjsmoothingcondR}
Let $\Omega$ be as in Theorem \ref{T:holconjsmoothing} and assume that Condition R holds.
Then for every $k \in \mathbb{N}$ there is a constant $C_{k} > 0$ such that
\begin{align}\label{E:holconjsmoothinginfinite}
\|Bf\|_{k} \leq C_{k}\|f\|
\end{align}
for $f$ conjugate holomorphic in $L^{2}(\Omega)$. In particular, $Bf$ is smooth up to the
boundary.
\end{corollary}

The proof of Theorem \ref{T:holconjsmoothing}, say again for $k_{1}=k_{2}=1$, also starts
with \eqref{E:Bf1norm}. But now, estimating $|(f,h)|$ means we are estimating
the (absolute value of) the integral of a {\it holomorphic} function (namely
$\overline{f}h$) over $\Omega$. Such an integral can be dominated by
$\|\overline{f}h\|_{-m}$, for any $m \in\mathbb{N}$ (\cite{Bell82b, BellBoas84, Komatsu84,
Straube84}). Finally, the equivalence, for holomorphic functions, of membership in a
negative Sobolev space and polynomial boundedness in the reciprocal of the boundary
distance (\cite{Bell82b, Straube84}) gives the estimate $\|\overline{f}h\|_{-m} \leq
C_{m,k} \|f\|\,\|h\|_{-k}$ for $m$ big enough (relative to $k$). We remark that the last
two steps are valid for functions in the kernel of more general elliptic operators
(systems), see \cite{LionsMagenes} and \cite{Roitberg},
respectively, for the relevant results. However, the first step after invoking
\eqref{E:Bf1norm} may fail: for example, the product of two harmonic functions need not be
harmonic.

\medskip

Regularity properties like \eqref{main1} are known to hold on a large class of domains.
When the domain $\Omega$ is pseudoconvex, these properties are essentially
equivalent to corresponding regularity properties of the $\overline{\partial}$-Neumann
operator. For example when the domain is of finite
type, or when it admits a defining function that is plurisubharmonic at boundary points
(in particular, when the domain is convex), \eqref{main1} holds for the pair $(k,k)$ for
all $k \in \mathbb{N}$. Nonpseudoconvex domains on which regularity estimates
for the Bergman projection hold include Reinhardt domains, complete Hartogs domains in
$\mathbb{C}^{2}$, and domains with `approximate symmetries'. For these results and for
further information on the $L^{2}$ Sobolev regularity theory of the Bergman projection
and of the $\overline{\partial}$-Neumann operator, we refer the reader to
\cite{ChenShaw01, BoasStraube99, Straube10} and their references. \cite{BoasStraube99}
also contains a discussion of the connection between the regularity theory of the Bergman
projection and duality of holomorphic function spaces.

\medskip
The paper is laid out as follows. In Section \ref{S:2} we define the notions
needed for
Theorem \ref{T:main}, collect some standard definitions, and give a new proof of (a
portion of) Bell's duality theorem on holomorphic functions in Sobolev spaces. In Section
\ref{S:3}, we state the anti-differentiation result, Proposition \ref{P:antiderivative},
then give a proof of Theorem \ref{T:main} modulo a proof of Proposition
\ref{P:antiderivative}. Section \ref{S:anti} is devoted to a proof of the
anti-differentiation result, broken down into several subsections. We develop the algebra
of operators associated to anti-differentiation along integral curves to transverse vector
fields in these subsections in some detail, as we need these results for our proof of
Proposition \ref{P:antiderivative}; we also mention that the results in Section
\ref{S:anti} can be applied elsewhere. In Section \ref{S:5} we prove Theorem
\ref{T:holconjsmoothing}.

\bigskip

\section{{Function spaces, duality, and complex transversality}}\label{S:2}

Throughout the paper, $\Omega$ will denote a domain with smooth boundary
$b\Omega$, contained in
$\mathbb{C}^{n}$. The standard $L^2$ inner product and norm on functions defined on
$\Omega$ will be denoted 
$$(f,g)=\int_\Omega f\, \bar g \qquad\text{and}\qquad \|f\|=\left(\int_\Omega |f|^2\right)^{1/2},$$
where the integrals are taken with respect to the Euclidean volume element.
If $k$ is a positive integer, we let $H^{k}(\Omega)$ denote the usual Sobolev space of complex-valued functions whose norm $\|.\|_k$
is induced by the inner product
$$(f,g)_k=\sum_{|\alpha|\leq k}\left(D^\alpha f, D^\alpha g\right),$$
for $\alpha$ a multi-index and $D^\alpha$ denoting differentiation of order $\alpha$. 
Let $C^{\infty}_{0}(\Omega)$ denote the set of smooth
functions with compact support in $\Omega$ and $C^{\infty}(\overline{\Omega})$ the set of 
functions smooth up to $b\Omega$. As is well-known,
$C^{\infty}(\overline{\Omega})$ is dense in $H^{k}(\Omega)$. 
We let $H_{0}^{k}(\Omega)$ be the closure of $C^{\infty}_{0}(\Omega)$
in $H^{k}(\Omega)$.

The dual space of $H_{0}^{k}(\Omega)$ will be denoted $H^{-k}(\Omega)$. Because
$C^{\infty}_{0}(\Omega)$ is dense in $H^{k}_{0}(\Omega)$, $H^{-k}(\Omega)$ imbeds
naturally into the space of distributions $\mathcal{D}^{\prime}(\Omega) =
\left(C^{\infty}_{0}(\Omega)\right)^{*}$ on $\Omega$, and $L \in
\mathcal{D}^{\prime}(\Omega)$ belongs to $H^{-k}(\Omega)$ precisely when 
\begin{equation}\label{E:dualnorm}
\|L\|_{-k}=\sup\left\{\left|\left\langle L,\phi\right\rangle\right|: 
\phi\in C^{\infty}_{0}(\Omega), \|\phi\|_{k} \leq 1\right\}
\end{equation}
is finite. Here $\left\langle L,\phi\right\rangle$ denotes the action of the distribution 
$L$ on the test function $\phi$. We shall only need to compute \eqref{E:dualnorm} on
certain $L\in L^2(\Omega)$, in which case $\left\langle L,\phi\right\rangle=(L,\bar\phi)$.

\medskip

The subspace of $H^{k}(\Omega)$ consisting of holomorphic functions will be denoted 
$A^{k}(\Omega)$. For consistency of notation, we let $A^{0}(\Omega)$ denote the Bergman
space: $A^{0}(\Omega)= L^2(\Omega)\cap{\mathcal O}(\Omega)$, where ${\mathcal O}(\Omega)$
denotes the space of holomorphic functions on 
$\Omega$. For functions $f$ in $A^{-k}(\Omega)$, the norm $\| f\|_{-k}$ is comparable to 
a weighted $L^2$ norms of $f$, with weight equal to the corresponding positive power of
the distance to $b\Omega$ (\cite{Ligocka86, Boas87}). If $r$ is a smooth
defining function for $\Omega$ -- i.e., $\Omega=\{r<0\}$ and $dr\neq 0$ when $r=0$ -- we
shall use the following version of (half of) this fact:
  for $k\in\mathbb{N}$ there exists a constant $\beta_{k}>0$ such that
  \begin{equation}\label{E:eqnorms2}
    \|(-r)^{k}h\|\leq\beta_{k}\|h\|_{-k}\qquad\sjump\forall\sjump h\in A^{0}(\Omega).
  \end{equation}
  
\medskip  
  
The Bergman projection is the orthogonal projection of $L^2(\Omega)$ onto
$A^{0}(\Omega)$. We say that $B$ satisfies {\it Condition R} (see \cite{BellLigocka}) if
it maps $C^{\infty}(\overline{\Omega})$ to itself (by continuity in $L^{2}(\Omega)$ and
the closed graph theorem for Fr\'{e}chet spaces, this map is then automatically
continuous). This property is often also referred to as {\it global regularity}. An
equivalent formulation is: for each $k_{2} \in \mathbb{N}$ there exists $k_{1} \in
\mathbb{N}$ such that $B$ maps $H^{k_{1}}(\Omega)$ (continuously) into
$H^{k_{2}}(\Omega)$. The special case where \eqref{main1}  holds with $k_{1}=k_{2}$ is
usually referred to as {\it exact regularity} at level $k=k_{1}=k_{2}$. We remark that,
rather intriguingly, no instance is known where $B$ satisfies Condition R, but is not
exactly regular at all levels.

\medskip
  
As mentioned in the introduction, a crucial fact for our proof of Theorem \ref{T:main}
is that regularity of the Bergman projection as in \eqref{main1} implies that the Sobolev
norm of a function $f\in A^{k_{2}}(\Omega)$ is controlled by pairing $f$ with a family of
holomorphic functions in the {\it ordinary} $L^2$ inner product.  Duality results of this
kind are known, \cite{Bell81a, Bell82a, Bell82b, Straube84, BellBoas84, Komatsu84,
Barrett95}, but they are  formulated for exact regularity, rather than \eqref{main1}, or
for Condition R (i.e., assuming \eqref{main1} holds for all $k_2\in\mathbb{N}$). In these
situations, a  clean cut formulation is
possible: for example, $A^{k}(\Omega)$ and $A^{-k}_{cl}(\Omega)$, the closure of
$A^{0}(\Omega)$ in $A^{-k}(\Omega)$, are mutually dual, under a natural extension of the
$L^{2}$ pairing, when $B$ is exactly regular at level $k$. When there is a loss of
derivatives, that is when $k_{2} < k_{1}$ in \eqref{main1}, this duality has a somewhat
less striking formulation. For this reason, we only state below what we need 
and give, for the reader's convenience, a straightforward proof (which seems to be new
even for the case $k_{1}=k_{2}$).
 
 \begin{proposition}\label{P:pieceofduality}
  Let $\Omega\subset\mathbb{C}^{n}$ be a domain with smooth boundary. Suppose that 
\eqref{main1} holds. Then there exists a constant $c>0$ such that
  \begin{align}\label{1stdual}
    \|f\|_{k_{2}}\leq c\sup\left\{\left|\left(f, h\right) \right|: h\in
A^{k_{2}}(\Omega), 
\|h\|_{-k_{1}}\leq 1\right \}
  \end{align}
   holds for all $f\in A^{0}(\Omega)$. The constant $c$ depends on $C_{1}$ from \eqref{main1},
$n$ and $k_{2}$.
\end{proposition}

\begin{remark}\label{dualitygenuine}
It is part of \eqref{1stdual} that if the right hand side is finite, then so is the left
hand side. That is, \eqref{1stdual} is a genuine estimate, as opposed to an \emph{a
priori} estimate.
\end{remark}

\begin{proof}
  Let $f\in A^{0}(\Omega) \subset C^{\infty}(\Omega)$ and let $\alpha$ be a fixed 
multi-index with $|\alpha|\leq k_{2}$. Then $D^{\alpha}f \in L^{2}(\Omega)$ if and only
if 
\begin{align*}
     \sup\left\{\left|\left( D^{\alpha}f,g \right) 
\right|: g\in C^{\infty}_{0}(\Omega),\sjump \|g\|\leq 1 \right\} < \infty \;.
\end{align*}
Moreover, in this case, $\|D^{\alpha}f\|$ is given by this supremum. Because $g\in
C_{0}^{\infty}(\Omega)$ and $f\in C^{\infty}(\Omega)$, integration by parts yields
  \begin{align*}
  \left|  \left( D^{\alpha}f,g\right)\right|=\left|(-1)^{|\alpha|}
\left( f, D^\alpha g\right)\right| 
    =\left|\left(Bf, D^\alpha g\right)\right| 
    =\left|\left(f, BD^\alpha g\right)\right|.
  \end{align*}
  Thus, 
  \begin{align}\label{E:Dalphafestimate}
     \left\|D^{\alpha}f\right\|=\sup\left\{\left|\left( f,BD^{\alpha}g \right) \right|: 
g\in C^{\infty}_{0}(\Omega),\sjump \|g\|\leq 1 \right\}.
      \end{align}
Note that \eqref{main1} implies that $BD^{\alpha}g \in A^{k_{2}}(\Omega)$.
The aim now is to see that \eqref{main1} forces $BD^{\alpha}g$ to be
uniformly bounded in $H^{-k_{1}}(\Omega)$. For that, let $\varphi \in C^{\infty}_{0}(\Omega)$ with
$\|\varphi\|_{k_{1}} \leq 1$. Integration by parts and Cauchy--Schwarz give
\begin{align}\label{est1}
   \left |\left(\varphi, BD^{\alpha}g\right)\right| = \left|\left(D^{\alpha}B\varphi,
  g\right)\right| \leq\|B\varphi\|_{k_{2}}\|g\| \leq
C_{1}\|\varphi\|_{k_{1}}\|g\|   
  \leq C_{1} \; ,
\end{align}
where \eqref{main1} was used again. Thus, \eqref{E:dualnorm} shows $\|BD^{\alpha}g\|_{-k_{1}}
\leq C_{1}$.  Returning to \eqref{E:Dalphafestimate}, we now obtain that
\begin{align*}
  \left\|D^{\alpha}f\right\|&=C_{1}\sup\left\{\left|\left(
f,BD^{\alpha}\frac{g}{C_{1}} \right) \right|: 
g\in C^{\infty}_{0}(\Omega),\sjump \|g\|\leq 1 \right\}\\
&\leq C_{1}\sup\left\{\left|\left( f,h\right) \right|: 
h\in A^{k_{2}}(\Omega),\sjump \|h\|_{-k_{1}}\leq 1 \right\}.
\end{align*}
Summing over $|\alpha|\leq k_{2}$ gives \eqref{1stdual} with $c=\sum_{j=0}^{k_{2}} {2n \choose j}C_{1}$.
 \end{proof}
 
\begin{remark}\label{dualityRemark}
 Proposition \ref{P:pieceofduality} implies that when \eqref{main1} holds, then
 \begin{align}\label{E:Bfknorm}
   \left\|Bf\right\|_{k_{2}}&\leq c\sup\left\{\left|(Bf,h)\right|: h\in
A^{k_{2}}(\Omega),
\sjump\|h\|_{-k_{1}}\leq 1
    \right\}\notag\\
    &=c\sup\left\{\left|(f,h)\right|: h\in A^{k_{2}}(\Omega),\sjump\|h\|_{-k_{1}}\leq 1
    \right\}
 \end{align}
 for all $f\in H^{k_{1}}(\Omega)$. It is \eqref{E:Bfknorm} that will be used in the proof
of Theorem \ref{T:main} (compare \eqref{E:Bf1norm} above).
\end{remark}

\medskip
Let $\Omega\ssubset\mathbb{C}^n=\{\rho<0\}$ be a  smoothly bounded domain and
denote by
$L_{n}:=$ \linebreak $\sum_{j=1}^{n}(\partial \rho/\partial \overline{z_{j}})(\partial /\partial
z_{j})$ the complex normal field of type $(1,0)$. Set $T_{0}= i(L_{n}-\overline{L_{n}})$.
Then $T_{0}$ is real and tangential to $b\Omega$, and it spans
(over $\mathbb{C}$) the orthogonal complement of $T^{1,0}(b\Omega) \oplus
T^{0,1}(b\Omega)$ in $T(b\Omega) \otimes \mathbb{C}$. In particular, a vector field $T$,
with coefficients in $C^{\infty}(\overline{\Omega})$, that is tangential to $b\Omega$
can be written as
\begin{equation*}
T = aT_{0} + Y_{1}+ \overline{Y_{2}},
\end{equation*}
where $a$ is a smooth complex-valued function, and
$Y_{1}$, $Y_{2}$ are of type $(1,0)$ and  tangential to $b\Omega$. The vector field $T$
is called {\it complex transversal}  if it is transversal to $T^{1,0}(b\Omega) \oplus
T^{0,1}(b\Omega)$, i.e.,  if $a$ is nowhere vanishing on $b\Omega$. Writing
$\overline{Y_{2}}$ as $a\left((1/a)\overline{Y_{2}}+(1/\overline{a})Y_{2}\right) -
(a/\overline{a})Y_{2}$, shows that
near $b\Omega$, $T$ can be written in the form
\begin{equation}\label{T-rep}
T = a\left(T_{1} + L\right),
\end{equation}
where $T_{1}$ is real and complex transversal,  $L$ is of type $(1,0)$, and both $T_{1}$ and $L$ are tangential to
$b\Omega$. 

\medskip

The Hilbert spaces $H_{T}^{k}(\Omega)$  that occur in Theorem \ref{T:main} are the usual
Sobolev spaces with respect to differentiation in the direction $T^{k}$, where the latter is the $k$-fold differentiation
with respect to $T$: $T^{k}(f)=T\left(T\left(\dots\left(Tf\right)\dots\right)\right)$. 
\begin{definition}\label{D:TSobolev}
 For $k\in\mathbb{N}$, set 
 $$H_{T}^{k}(\Omega)=\left\{f\in L^{2}(\Omega):  T^{j}f  \in L^{2}(\Omega),\;j\in\{1,\dots,k\}\right\},$$ 
 where $T^j f$ is taken in the sense of distributions. 
\end{definition}
For fixed $k$, $H_{T}^{k}(\Omega)$ is a Hilbert space with respect to the
inner product
\begin{align*}
  (f,g)_{k,T}:=\sum_{j=0}^{k}\left(T^{j}f,T^{j}g\right)\qquad\sjump\forall\sjump f,g\in 
H_{T}^{k}(\Omega),
\end{align*}
and $C^{\infty}(\overline{\Omega})$ is dense in $H_{T}^{k}(\Omega)$ with respect to the norm 
$$\|f\|_{k,T}^{2}=\sum_{j=0}^{k}\|T^{j}f\|^{2}$$
induced by this inner product.
The completeness of $H_{T}^{k}(\Omega)$ is proved in the same way as for ordinary Sobolev
spaces; density of $C^{\infty}(\overline{\Omega})$ is a standard application of
Friedrichs' Lemma (see for example \cite{ChenShaw01}, Lemma D.1 and Corollary D.2). The spaces
$H_{T}^{k}(\Omega)$, $k\geq 1$, depend on the choice of the tangential, complex transversal vector field $T$; cf. Section 5 in
\cite{HerMcN10}. However, if  $b$ is a  smooth, non-vanishing, complex-valued function,
then $H_{T}^{k}(\Omega)$ and $H_{bT}^{k}(\Omega)$ are equal,
and an inductive argument gives, e.g., the somewhat rough estimate
\begin{align}\label{E:TbTchange}
  \left\|(bT)^{k}f\right\|\leq b_{k}\|f\|_{k,T}\qquad\text{for}\sjump
b_{k}=
\max_{\overline{\Omega}, 0\leq \ell \leq k}\left\{\left|T^{\ell}b \right|^{k} ,1\right\}.
\end{align}
This implies in particular that Theorem \ref{main1} holds for $T$ if and only if it holds 
for $bT$. Moreover, \eqref{E:TbTchange} indicates how the
constant in \eqref{main2} changes (although there are other quantities that determine the
constant in \eqref{main2}).

The Fr\'{e}chet space 
$H_{T}^{\infty}(\Omega) =
\cap_{k=0}^{\infty}H_{T}^{k}(\Omega)$ is equipped with the (metrizable) topology induced by the
family of norms $\{\|.\|_{k,T} :  k \in \mathbb{N}\}$ (see e.g. \cite{Rudin91}). It
inherits completeness from the $H_{T}^{k}(\Omega)$. For a class of examples of the fact that  
$ C^{\infty}(\overline{\Omega})\subsetneq H_{T}^{\infty}(\Omega)$, see Section 5
in \cite{HerMcN10}.

\bigskip

\section{{Proof of Theorem \ref{T:main}}}\label{S:3}

As indicated in the introduction, the second step in proving Theorem \ref{T:main}  relies on a representation of a
holomorphic function $h$ in terms of $\overline{T}^{j}$-derivatives, $j\leq k$, of functions whose $L^{2}$ norms are
controlled by $\|h\|_{-k}$ for $k \in \mathbb{N}$.

\begin{proposition}\label{P:antiderivative}
Let $\Omega\ssubset\mathbb{C}^{n}$ be a smoothly bounded  domain and let $T$ be
a vector field which is tangential to $b\Omega$  and  complex transversal. Let
$k\in\mathbb{N}$.

There exists an open neighborhood $U$ of $b\Omega$, a function $\zeta\in 
C^{\infty}_{0}(\overline{\Omega}\cap U)$ which equals $1$ near $b\Omega$, and constants
$C_{j,k}>0$ for $j\in\{0,\dots,k\}$, such that for all $h\in A^{0}(\Omega)$ there exist
functions $\mathcal{H}_{j}^{k} \in H^{k}(\Omega) \cap C^{\infty}(\Omega)$,
$j\in\{0,\dots,k\}$, on $\Omega$
satisfying

\begin{itemize}
  \item[(i)] $\zeta h=\sum_{j=0}^{k}\overline{T}^{j}\mathcal{H}_{j}^{k}$ on $\Omega$,
  
  \smallskip
  
  \item[(ii)] $\|\mathcal{H}_{j}^{k}\|\leq C_{j,k}\|h\|_{-k}$ for all $j\in\{0,\dots,k\}$.
\end{itemize}
\end{proposition}

\begin{remark}
For simplicity, we have only stated what we need in Proposition \ref{P:antiderivative}.
As usual in such situations, membership in $H^{k}(\Omega)$ actually comes with a norm
estimate that can be worked out with a little additional care. In fact, the arguments in
\S\ref{SS:antiderivative} can be refined to show that for some constant $C_{s,k}>0$,
$\|\mathcal{H}^{k}_{j}(h)\|_{s} \leq C_{s,k} \|h\|_{s-k}$, for $s \in \mathbb{R}$.
\end{remark}

The proof of Proposition \ref{P:antiderivative} will be given in Section \ref{S:anti}.
Assuming it is true, we now give the proof of Theorem \ref{T:main}. 
We will use the notation $A\lesssim B$ to denote the existence of a constant $C$
independent of $f$, but allowed to depend on $\Omega$ and $T$, such that $A\leq C\cdot B$.

\begin{proof}[Proof of Theorem \ref{T:main}]

Assume first
$f\in C^\infty\left(\overline{\Omega}\right)$. Then $Bf \in A^{k_{2}}(\Omega)$, by
\eqref{main1}. 
Inequality \eqref{E:Bfknorm} states that
  \begin{align}\label{E:Bfknorm1}
        \left\|Bf\right\|_{k_{2}}\leq c\sup\left\{\left|(f,h)\right|: h\in
A^{k_{2}}(\Omega),\sjump\|h\|_{-k_{1}}\leq 1 \right\}
 \end{align}
 holds for some constant $c>0$ independent of $f$.
 
We use Proposition \ref{P:antiderivative} for $k=k_{1}$ to estimate the right-hand side of
\eqref{E:Bfknorm1} as follows.
For  $T$ as in Theorem \ref{main1},  choose the neighborhood $U$ of $b\Omega$ 
and the cut-off function $\zeta$ described in Proposition \ref{P:antiderivative}.
Using the partition of unity $\{\zeta,1-\zeta\}$ yields
 \begin{align}\label{A}
     \left|\left(f,h\right)\right|&\leq\left|\left(f,\zeta h\right)\right|+\left|
\left(f,(1-\zeta)h\right)\right|\\
     &=\Bigl|\bigl(f,
\sum_{j=0}^{k_{1}}\overline{T}^{j}\mathcal{H}_{j}^{k_{1}}\bigr)\Bigr|+\left|
\left(f,(1-\zeta)h\right)\right|. \notag
 \end{align}
As $(1-\zeta)$ is identically zero near $b\Omega$, it follows from the  
Cauchy--Schwarz inequality, \eqref{E:eqnorms2}, and $\|h\|_{-k_{1}}\leq 1$  that
 \begin{align}\label{B}
   \left|\left(f,(1-\zeta)h\right)\right|\leq \widetilde{\beta}_{k}\|f\|\cdot\|h\|_{-k_{1}}\leq
   \widetilde{\beta}_{k}\|f\|,
 \end{align}
where the constant $\widetilde{\beta}_{k}>0$ depends on $\beta_{k}$ (see \eqref{E:eqnorms2}) and $\zeta$ . 
Since $T$ is tangential, integration by
parts (justified since $\mathcal{H}_{j}^{k_{1}} \in H^{k_{1}}(\Omega)$) gives
 \begin{align*}
   \left|\left( f, \overline{T}^{j}\mathcal{H}_{j}^{k_{1}}\right)\right|= \left|
\left(Tf, \overline{T}^{j-1}\mathcal{H}_{j}^{k_{1}}\right)+\left(g_{0}f,
\overline{T}^{j-1}\mathcal{H}_{j}^{k_{1}}\right)\right|,
 \end{align*}
where $g_{0}$, depending on the first order derivatives of the coefficients of $T$, is a smooth function on $\overline{\Omega}$.  
Continuing this procedure yields
 \begin{align*}
  \left|\left(f, \overline{T}^{j}\mathcal{H}_{j}^{k_{1}}\right)\right|\leq
    \sum_{\ell=0}^{j}\left\|g_{\ell}T^{\ell}f\right\|\cdot\|\mathcal{H}_{j}^{k_{1}}\|
 \end{align*}
for functions $g_{\ell}\in C^{\infty}(\overline{\Omega})$, $\ell\in\{0,\dots,j-1\}$ and 
$g_{j}=1$. By part (ii) of Proposition \ref{P:antiderivative} it now follows that
 \begin{align*}
    \left|\left( f, \overline{T}^{j}\mathcal{H}_{j}^{k_{1}}\right)\right|\leq C_{j,k}
\sum_{\ell=0}^{j}\left\|g_{\ell}T^{\ell}f\right\|.
 \end{align*}
 Summing over $j\leq k_{1}$ gives, in view of \eqref{E:Bfknorm1}, \eqref{A}, and
\eqref{B} that
\begin{align}\label{Cinfty}
    \left\|Bf\right\|_{k_{2}}\leq C\|f\|_{k_{1},T}\sjump\qquad\forall\sjump f \in
C^{\infty}(\overline{\Omega}).
\end{align}

To remove the smoothness assumption on $f$, approximate $f$ with respect to
$\|.\|_{k_{1},T}$ by functions in $C^{\infty}(\overline{\Omega})$. Then invoke
\eqref{Cinfty} and the continuity of $B$ in $L^{2}(\Omega)$ to obtain \eqref{main2}; for
more details see Lemma 4.2 in \cite{HerMcN10}.
\end{proof}


\section{{Anti-differentiation along a transverse direction}}\label{S:anti}

\subsection{General set-up}\label{SS:setup}
In this and the next two sections we work in the real setting of $\mathbb{R}^{n}$. Let
$\Omega$ be a bounded domain with smooth boundary. Denote by
$\mathcal{N}=\sum_{j=1}^{n}\mathcal{N}_{j}\frac{\partial}{\partial x_{j}}$ a vector
field with smooth coefficients in a neighborhood of $\Omega$. Suppose that $\mathcal{N}$
is transversal to $b\Omega$. Then there exists an open, bounded neighborhood
$V\subset\mathbb{R}^{n}$ of $b\Omega$ on which $\mathcal{N}$ is non-vanishing. Moreover,
for each $x\in V$ there exist a scalar $\tau_{x}>0$ and a smooth integral curve
$\varphi_{x}: (-\tau_{x},\tau_{x})\longrightarrow\mathbb{R}^{n}$ satisfying
$\varphi_{x}(0)=x$ and
\begin{align*}
  \frac{\partial\varphi_{x}}{\partial t}(t)=\Bigl\langle
\mathcal{N}_{1}\left(\varphi_{x}(t)\right),
\dots,\mathcal{N}_{n}\left(\varphi_{x}(t)\right)\Bigr\rangle
\end{align*}
for all $t\in(-\tau_{x},\tau_{x})$. Because the curve $\varphi_{x}$ intersects $b\Omega$
transversally, there exists a scalar $\tau_{0}>0$ such that $\varphi_{p}$ is defined on
$(-\tau_{0},\tau_{0})$ for all $p\in b\Omega$. After possibly rescaling $\mathcal{N}$, it
may
be assumed that $\tau=2$.

For each $x\in V$ define $t_{x}$ to be the (unique) scalar for which $\varphi(t_{x},x)\in
b\Omega$, and note that $|t_{x}|=\mathcal{O}(d_{b\Omega}(x))$.  It may be assumed that
$t_{x}>0$ for all $x\in V\cap\Omega$. Set $U:=\{x : |t_{x}|<1\}$. Then $U$ is an open neighborhood
of $b\Omega$. Moreover, the flow map, $\varphi(t,x):=\varphi_{x}(t)$, is 
a smooth map on $(-1,1)\times U$ satisfying $\varphi(0,x)=x$ as well as
\begin{align}\label{D:integralcurve}
  \frac{\partial\varphi_{\ell}}{\partial
t}(t,x)=\mathcal{N}_{\ell}\left(\varphi(t,x)\right)
\qquad\sjump\forall\sjump\ell\in\{1,\dots,n\}.
\end{align}

\medskip

We can now make precise the anti-differentiation operator $\mathfrak{A}$ from Section
\ref{intro}. Denote by $C^{\infty}_{\overline{U}}(\Omega) := \{f \in
C^{\infty}(\Omega)\,|\,f \equiv 0 \;\text{on}\; \Omega \setminus \overline{U}\}$. Then we
define $\mathfrak{A}$ as an operator from $C^{\infty}_{\overline{U}}(\Omega)$ to itself:
\begin{align*}
\mathfrak{A}[g](x)=\left\{\begin{array}{cl}
                          \int_{-1}^{0}(g\circ\varphi)(s,x)\,ds
                          &\qquad\text{if}\sjump x\in\Omega\cap U,\\
                           0  &\qquad\text{if}\sjump x \in \Omega \setminus U.
                          \end{array} \right.
\end{align*}
$\mathfrak{A}$ belongs to a class of operators denoted by $\mathcal{A}^{1}_{0,0}$ below;
see Definition \ref{def1}, where the mapping properties are also discussed.
$\mathfrak{A}$ inverts $\mathcal{N}$ in the following sense.

\begin{lemma}\label{L:FTC}
  Let $g \in C^{\infty}_{\overline{U}}(\Omega)$. Then
  \begin{align}\label{E:FTC}
    g(x)=\mathfrak{A}[\mathcal{N}g](x)\sjump, \sjump x\in\Omega.
  \end{align}
\end{lemma}

\begin{proof}
When $x \in \Omega \setminus U$, both sides of \eqref{E:FTC} equal zero. When $x \in
\Omega \cap U$, then  $g\left(\varphi(-1,x)\right)=0$ because $\varphi(-1,x) \in
\Omega \setminus U$. Note that \eqref{D:integralcurve} implies 
$$\left(\left(\mathcal{N}g\right)\circ\varphi\right)(s,x) =
\frac{\partial}{\partial s}\left(\left(g\circ\varphi\right)(s,x)\right),$$ hence
an application of the Fundamental Theorem of
Calculus completes the proof of  \eqref{E:FTC}.
\end{proof}


\subsection{The spaces $\mathcal{A}_{*,*}^{*}$ and their $L^{2}$ mapping
properties.}\label{SSspaces}
In order to organize the proof of Proposition \ref{P:antiderivative}, we now define
several spaces of operators related to, but more general than $\mathfrak{A}$. We group
the operators both according to their form as well as according to their mapping
properties. 

\begin{definition}\label{def1}
(1) Denote by $\mathcal{A}_{\mu,0}^{1}$, $\mu\in\mathbb{N}_{0}$ the space of
operators acting on $C^{\infty}_{\overline{U}}$ which are of the form
  \begin{align}\label{basicHardy}
A[g](x)=\left\{\begin{array}{cl}
               \int_{-1}^{0}s^{\mu}\gamma_{A}(s,x)
               \cdot(g\circ\varphi)(s, x)\,ds &\qquad\text{if}\sjump x\in \Omega\cap U,\\
                                                           \\
               0&\qquad\text{if}\sjump x\in \Omega\setminus U
               \end{array} \right.
 \end{align}
for some $\gamma_{A}\in C^{\infty}([-1,0]\times \overline{U})$. Then $A[g] \in
C^{\infty}_{\overline{U}}(\Omega)$. 

\noindent (2) Denote by $\mathcal{A}_{\mu,\nu}^{1}$,  $\mu\in\mathbb{N}_{0}$,
$\nu\in\mathbb{N}$, the space of operators on $C^{\infty}_{\overline{U}}(\Omega)$
spanned by operators of the form $A_{\mu}\circ D^{\beta}$ for
$A_{\mu}\in\mathcal{A}_{\mu,0}^{1}$ and $|\beta|\leq\nu$.  
\end{definition}
These operators have mapping properties in weighted $L^{2}$ spaces that will be very
useful for our purposes. We introduce the following definition: 

\begin{definition}\label{D:Sjkspace}
  An operator $A$ on $C^{\infty}_{\overline{U}}(\Omega)$ is said to belong to
$\mathcal{S}_{\nu}^{k}$ if there exists a constant $C>0$ such that
\begin{align}\label{E:Sjkspace}
  \left\|t_{x}^{\ell}\cdot A[g]\right\|\leq C\sum_{|\beta|\leq
\nu}\left\|t_{x}^{\ell+k}\cdot 
D^{\beta}g\right\|\qquad\forall\sjump \ell\in\mathbb{N}_{0},\sjump g\in
C^{\infty}_{\overline{U}}(\Omega).
\end{align}  
Here, $C$ does not depend on $g$ or
$\ell$. Note that because $t_{x}$ is defined on $\Omega \cap \overline{U}$ and both $g$
and $A[g]$ vanish on $\Omega \setminus \overline{U}$, both sides of \eqref{E:Sjkspace} are
well-defined.
\end{definition}

The following lemma is the key to the mapping properties of the anti-derivative
operator $\mathfrak{A}$ and its generalizations; in particular, it makes precise the
notion of `gaining' a factor of the boundary distance. As mentioned in Section
\ref{intro}, it is a consequence of one of a group of inequalities due to Hardy
(\cite{HLP}, sections 9.8, 9.9). 

\begin{lemma}\label{P:Hardy}
For given $\mu\in\mathbb{N}_{0}$, define
the operator $B_{\mu}$ on $C^{\infty}_{\overline{U}}(\Omega)$ by
  \begin{align}\label{B-mu}
    B_{\mu}[g](x)=
\left\{\begin{array}{cl}
\int_{-1}^{0}(t_{\varphi(s,x)})^{\mu}\cdot\left|g(\varphi(s,x) \right|\,
ds &\qquad\text{if}\sjump x\in \Omega\cap U \\
0&\qquad\text{if}\sjump x \in \Omega \setminus U.
      \end{array}\right.
\end{align}
  Then $B_{\mu}\in\mathcal{S}_{0}^{\mu+1}$.
\end{lemma}

\begin{proof}
It is convenient in this proof (but not later on) to rewrite $B_{\mu}$ using coordinates
$(\tau, p) \in [0,2]\times b\Omega$ on $\overline{\Omega}\cap\overline{V}$ given by
$(\tau,p) \rightarrow \varphi(-\tau,p)$. Expressing $B_{\mu}$ in these coordinates and
changing variables gives
\begin{align}\label{B-special}
B_{\mu}[g](\tau,p)= \int_{\tau}^{\tau +1}\sigma^{\mu}\left|g(\sigma,p)\right|\,d\sigma
\;\;,\;\; (\tau,p)\sim x = \varphi(-\tau,p) \in \Omega\cap U.
\end{align}
Estimate \eqref{E:Sjkspace} now follows from the following inequality of
Hardy's (see \cite{HLP}, Theorem 330, section 9.9) for $f \geq 0$:
\begin{align}\label{Hardy}
\int_{0}^{\infty}\left(x^{r}\int_{x}^{\infty}f(t)\,dt\right)^{2}\,dx \leq
\frac{4}{(2r+1)^{2}}\int_{0}^{\infty}\left(t^{r+1}f(t)\right)^{2}\,dt \; , \; r >
-\frac{1}{2}.
\end{align}
Indeed, to verify \eqref{E:Sjkspace} for $B_{\mu}$, it suffices to observe that for $x \sim (\tau,p)
\in \Omega\cap U$
\begin{align}
\left|\tau^{\ell}B[g](\tau,p)\right| \leq
\tau^{\ell}\int_{\tau}^{2}\sigma^{\mu}\left|g(\sigma,p)\right|\,d\sigma,
\end{align}
then to apply \eqref{Hardy} with $f(t)=t^{\mu}\left|g(t,p)\right|\;,\; 0\leq t\leq 2$, and
$f(t)=0\;,\;t>2$, $r = \ell$, for $p \in b\Omega$ fixed, and lastly, to integrate over
$b\Omega$. We also use the equivalence, with a constant that depends only on $\Omega$ and
$\mathcal{N}$, of the volume elements $dV$ and $d\tau\times d_{b\Omega}$ on
$\Omega\cap\overline{V}$, where $dV$ and $d_{b\Omega}$ denote the Euclidean volume
elements on $\mathbb{R}^{n}$ and $b\Omega$, respectively. Also note that for $r=\ell \geq
0$, the constant on the right hand side of \eqref{Hardy} is less than or equal to $4$.
Replacing it by $4$ yields a constant in \eqref{E:Sjkspace} that does not
depend on $\ell$.
\end{proof}

If $A$ is of the form $A=A_{\mu}\circ D^{\beta}$, then the inequality $|s| \leq |s|+t_{x}=
t_{\varphi(s,x)}$ and the boundedness of $|\gamma_{A}|$ on $[-1,0] \times
\overline{U}$ imply
\begin{align}\label{weighted}
\|t_{x}^{\ell}A[g]\| = \|t_{x}^{\ell}A_{\mu}[D^{\beta}g]\| \lesssim
\|t_{x}^{\ell}B_{\mu}[D^{\beta}g]\| \lesssim \|t_{x}^{\ell + \mu + 1}D^{\beta}g\|,
\end{align}
where the last inequality in \eqref{weighted} follows from Lemma \ref{P:Hardy}. This proves:

\begin{lemma}\label{C:A1estimate}
  $\mathcal{A}_{\mu,\nu}^{1}\subset\mathcal{S}_{\nu}^{\mu+1}$, $\mu\,,\,\nu \in
\mathbb{N}_{0}$.
\end{lemma}

We will also need notation for compositions of operators:

 \begin{definition}  
  \noindent(1) For a multi-index
$\alpha=(\alpha_{1},\dots,\alpha_{\ell})\in\mathbb{N}_{0}^{\ell}$, 
$\ell\in\mathbb{N}$; define 
  \begin{align*}
     \mathcal{A}_{\alpha,0}^{\ell}=\left<A_{1}\circ\dots\circ A_{\ell} : 
A_{j}\in\mathcal{A}_{\alpha_{j},0}^{1}\right\rangle \; .
  \end{align*}
  
\noindent(2) Denote by $\mathcal{A}_{\alpha,\nu}^{\ell}$,
$\alpha\in\mathbb{N}_{0}^{\ell}$, 
$\nu\in\mathbb{N}$ and $\ell\in\mathbb{N}$, the space of operators on
$C^{\infty}_{\overline{U}}(\Omega)$ spanned by operators of the form
$A_{\alpha}^{\ell}\circ D^{\beta}$ for $A_{\alpha}^{\ell}\in\mathcal{A}_{\alpha,0}^{\ell}$
and $|\beta|\leq\nu$. 
\end{definition}

With Lemma \ref{C:A1estimate} above, we have the following weighted mapping properties
of these operators:
\begin{lemma}\label{C:Aellestimate}
  $\mathcal{A}_{\alpha,\nu}^{\ell}\subset\mathcal{S}_{\nu}^{\ell+|\alpha|}$, where 
$|\alpha|=\sum_{j=1}^{\ell}\alpha_{j}$.
\end{lemma}


\subsection{On the algebra of commutators of $\mathcal{A}_{*,*}^{*}$ and differential
operators}\label{SScommutators}

Throughout this subsection, $X = \sum_{j=1}^{n}X_{j}\frac{\partial}{\partial x_{j}}$
denotes a vector field with smooth coefficients on (a neighborhood of)
$\overline{\Omega}$. We collect a series of elementary lemmas that give control over
various commutators involving the operators introduced in the previous subsection and $X$
or $D^{\beta}$. The notation $+$, or $\sum$, will be used to indicate sums of operators in
the indicated spaces.

\begin{lemma}\label{P:basiccommutator}
 Let $A\in\mathcal{A}_{\mu,0}^{1}$ for some $\mu\in\mathbb{N}_{0}$. Then
  \begin{align}\label{E:basiccommutator}  
    [A,X]\in \mathcal{A}_{\mu,0}^{1} + \mathcal{A}_{\mu+1,1}^{1}.
  \end{align}  
\end{lemma}

\begin{proof}
 Let  $g\in C^{\infty}_{\overline{U}}(\Omega)$. Then
  \begin{align}\label{E:XAterm}  
     (X\circ A)[g](x)
=\int_{-1}^{0}s^{\mu}\Bigl(\left(X_{x}(\gamma_{A})\right)\cdot 
(g\circ\varphi)\Bigr)&(s,x)\;ds\\
     +\int_{-1}^{0}s^{\mu}&\gamma_{A}(s,x)\cdot
X_{x}\left(g\circ\varphi)(s,x)\right)\;ds.
\notag
  \end{align}
Let $A_{0}$ be the operator in $\mathcal{A}_{\mu,0}^{1}$ such that the first term on
the right hand side of \eqref{E:XAterm} equals $-A_{0}[g](x)$. Then it follows that
\begin{align}\label{commAX}
[A,X][g](x)=A_{0}[g](x)+\int_{-1}^{0}s^{\mu}\gamma_{A}(s,
x)\Bigl((Xg)\circ
\varphi-X_{x}(g\circ\varphi)\Bigr)(s,x) \;ds.
\end{align}  
$A_{0}$ belongs to $\mathcal{A}^{1}_{\mu,0}$. That the term on the right
hand side of \eqref{commAX} given by the integral belongs to $\mathcal{A}^{1}_{\mu+1,1}$
can be seen as follows. Because $\varphi(0,x) \equiv x$, $\partial \varphi_{j}/\partial
x_{k} = \delta_{j,k} + O(s)$. Therefore, the chain rule shows that
$\Bigl((Xg)\circ\varphi-X_{x}(g\circ\varphi)\Bigr)(s,x)$ is a sum of terms each of which
is of the form $s\gamma_{\beta}(s,x)D^{\beta}g(\varphi(s,x))$, where $\gamma_{\beta}$ is smooth (depending on first derivatives of $X\circ\varphi$ and second derivatives of $\varphi$) and
$|\beta|=1$. 
\end{proof}

\begin{lemma}\label{L:basiccommutator1}
 Let   
$A_{\mu,\nu}\in\mathcal{A}_{\mu,\nu}^{1}$ for some $\mu\in\mathbb{N}_{0}$,
$\nu\in\mathbb{N}$.  Then
  \begin{align}\label{E:basiccommutator1}
    [A_{\mu,\nu},X]\in\mathcal{A}_{\mu,\nu}^{1} + \mathcal{A}_{\mu+1,\nu+1}^{1}.
  \end{align}
 \end{lemma}
\begin{proof}
  By definition of $\mathcal{A}_{\mu,\nu}^{1}$, $A_{\mu,\nu}$ can be written as a linear 
combination of operators of the form $A_{\mu}\circ D^{\beta}$, where
$A_{\mu}\in\mathcal{A}_{\mu,0}^{1}$ and $|\beta|\leq\nu$. It suffices to show
\eqref{E:basiccommutator1} for the latter operators. For that, note that
  \begin{align*}
  [A_{\mu}\circ D^{\beta},X]&=A_{\mu}\circ D^{\beta}\circ X-X\circ A_{\mu}\circ D^{\beta}\\
  &=A_{\mu}\circ[D^{\beta},X]+A_{\mu}\circ X\circ D^{\beta}-X\circ A_{\mu}\circ D^{\beta}\\
  &=A_{\mu}\circ[D^{\beta},X]+[A_{\mu},X]\circ D^{\beta}.
\end{align*} 
Since $[D^{\beta},X]$ is a differential operator of order $|\beta|-1$,
$A_{\mu}\circ[D^{\beta},X]$ 
belongs to $\mathcal{A}_{\mu,|\beta|-1}^{1}$. Furthermore, \eqref{E:basiccommutator}
yields that $[A_{\mu},X]\circ
D^{\beta}\in\mathcal{A}_{\mu,|\beta|}^{1} + \mathcal{A}_{\mu+1,|\beta|+1}^{1}$.
\end{proof}

Lemma \ref{P:basiccommutator} extends to iterated commutators as follows. Set
$C_{X}^{\nu}(A)=\left[C_{X}^{\nu-1}(A),X\right]$ with $C_{X}^{1}(A)=[A,X]$. That is,
$C_{X}^{\nu}(A)$ equals the $\nu$-fold iterated commutator $[\cdots[A,X],X],\cdots X]$.
\begin{lemma}\label{C:basiccommutator1}
Let  $A$ be given as in Lemma \ref{P:basiccommutator}.
Then 
$C_{X}^{\nu}(A)\in \sum_{j=0}^{\nu}\mathcal{A}_{\mu+j,j}^{1}$.
\end{lemma}
\begin{proof}
The proof is done via induction on $\nu$. Note first that the case $\nu=1$
is Lemma
\ref{P:basiccommutator}.  Next suppose that
$C_{X}^{\nu-1}(A)\in\sum_{j=0}^{\nu-1}\mathcal{A}_{\mu+j,j}^{1}$ holds for some 
$\nu\in\mathbb{N}$. It follows from Lemma 
\ref{L:basiccommutator1} that  $[A_{\mu+j,j},X]\in\mathcal{A}_{\mu+j,j}^{1} +
\mathcal{A}_{\mu+j+1,j+1}^{1}$ for $A_{\mu+j,j}\in\mathcal{A}_{\mu+j,j}^{1}$, which
completes the proof.
\end{proof}

We next consider commutators with higher order derivatives.
\begin{lemma}\label{L:basiccommutator2}
If $A\in\mathcal{A}_{\mu,0}^{1}$ for some $\mu \in \mathbb{N}_{0}$, then 
$\left[A,D^{\beta}\right]\in  \mathcal{A}_{\mu,|\beta|-1}^{1}
  + \mathcal{A}_{\mu+1,|\beta|}^{1}$.
\end{lemma}
\begin{proof}
The proof is done via induction on $|\beta|$. The case $|\beta|=1$ follows
from Lemma \ref{P:basiccommutator}. Suppose now that
$\left[A,D^{\sigma}\right]\in\mathcal{A}_{\kappa,\nu-2}^{1}\cup\mathcal{A}_{\kappa+1,\nu-1
}^{1}$ 
holds for any multi-index $\sigma$ of length $\nu-1$ for some $\nu\in\mathbb{N}$ and for
all $\kappa\in\mathbb{N}_{0}$. Let $\beta$ be a multi-index of length $\nu$. Then
$D^{\beta}$ may be written as $D^{\beta-\sigma}\circ D^{\sigma}$ for some
multi-index $\sigma$ of length $\nu-1$. It then follows that
  \begin{align*}
    [A,D^{\beta}]&=[A, D^{\beta- \sigma}]\circ D^{\sigma} + D^{\beta -
\sigma}\circ A\circ D^{\sigma} - D^{\beta - \sigma}\circ D^{\sigma}\circ A\\
    &=[A, D^{\beta - \sigma}]\circ D^{\sigma} + D^{\beta - \sigma}\circ [A,
D^{\sigma}]\,.
  \end{align*}
It follows from Lemma \ref{P:basiccommutator} and the fact that $|\beta-\sigma|=1$ that 
$\left[A,D^{\beta - \sigma}\right]\circ
D^{\sigma}\in\mathcal{A}_{\mu,|\beta|-1}^{1} +
\mathcal{A}_{\mu+1,|\beta|}^{1}$. The induction hypothesis furnishes two operators
$A_{1} \in \mathcal{A}^{1}_{\mu,\nu -2}$ and $A_{2} \in \mathcal{A}^{1}_{\mu +1,\nu -1}$
so that $\left[A,D^{\sigma}\right] = A_{1}+A_{2}$. Then 
\begin{align*}
D^{\beta - \sigma}\circ \left[A,D^{\sigma}\right] &= D^{\beta -
\sigma}\circ (A_{1}+A_{2}) \\ 
&= A_{1}\circ D^{\beta - \sigma} + A_{2}\circ D^{\beta - \sigma} +
[D^{\beta - \sigma},A_{1}] + [D^{\beta - \sigma},A_{2}]\; .
\end{align*}
The four terms on the right hand side are in, respectively, $\mathcal{A}^{1}_{\mu,\nu
-1}$, $\mathcal{A}^{1}_{\mu+1,\nu}$, $\mathcal{A}^{1}_{\mu,\nu-2} +
\mathcal{A}^{1}_{\mu+1,\nu-1}$, and $\mathcal{A}^{1}_{\mu+1,\nu-1} +
\mathcal{A}^{1}_{\mu+2,\nu}$. We have used Lemma \ref{L:basiccommutator1} for the last two
terms ($|\beta - \sigma|=1$!). Taking into account the (trivial) inclusions
$\mathcal{A}^{1}_{\mu,\nu-2} \subset \mathcal{A}^{1}_{\mu,\nu-1}$,
$\mathcal{A}^{1}_{\mu+1,\nu-1} \subset\mathcal{A}^{1}_{\mu,\nu-1}$, and
$\mathcal{A}^{1}_{\mu+2,\nu} \subset \mathcal{A}^{1}_{\mu+1,\nu}$, shows that $D^{\beta -
\sigma}\circ \left[A,D^{\sigma}\right] \in \mathcal{A}^{1}_{\mu,\nu-1} +
\mathcal{A}^{1}_{\mu+1,\nu}$. This completes the induction.
\end{proof}

For compositions of operators, Lemma \ref{L:basiccommutator2} takes on the following form.

\begin{lemma}\label{C:basiccommutator3}
 If $A_{\alpha,\nu}\in\mathcal{A}_{\alpha,\nu}^{\ell}$ for some multi-index 
$\alpha\in\mathbb{N}_{0}^{\ell}$, $\nu\,,\,\ell\in\mathbb{N}$, then 
  \begin{align}\label{E:basiccommutator3}
  [A_{\alpha,\nu},D^{\beta}]\in\mathcal{A}_{\alpha,\nu+|\beta|-1}^{\ell}
  +\sum_{j=1}^{\ell}\mathcal{A}_{\alpha+e_{j},\nu+|\beta|}^{\ell}, 
 \end{align} 
  where $e_{j}$ is the standard $j$-th unit vector.
\end{lemma}
\begin{proof}
We first consider the case $\nu=0$; in this case, the proof is by induction on $\ell$. The
case $\ell=1$ is Lemma \ref{L:basiccommutator2}. The induction step is analogous to that
in
the proof of Lemma \ref{L:basiccommutator2}, with the multi-index $\alpha$ now playing the
role of $\beta$ there. We leave the details to the reader.

When $\nu>0$, $A_{\alpha,\nu}$ is a linear combination of operators of the form
$A_{\alpha,0}\circ D^{\gamma}$ with $|\gamma| \leq \nu$. Thus $[A_{\alpha,\nu},
D^{\beta}]$ is a linear combination of terms of the form
\begin{align*}
[A_{\alpha,0}\circ D^{\gamma}, D^{\beta}] =
A_{\alpha,0}\circ D^{\gamma}\circ D^{\beta} - D^{\beta}\circ A_{\alpha,0}\circ D^{\gamma}
= [A_{\alpha,0}, D^{\beta}]\circ D^{\gamma}
\end{align*}
because $D^{\gamma}\circ D^{\beta} = D^{\beta}\circ D^{\gamma}$.
\eqref{E:basiccommutator3} now follows from the case $\nu=0$ (already shown) applied to
$[A_{\alpha,0}, D^{\beta}]$.
\end{proof}


\subsection{Proof of Proposition \ref{P:antiderivative}}\label{SS:antiderivative}

Let $\Omega$ and $T$ be as in Proposition \ref{P:antiderivative}. Near $b\Omega$,
$T=a(T_{1}+L)$, with both $T_{1}$ and $L$ tangential, $T_{1}$ real, $L$ of type $(1,0)$,
and $a$ a smooth function that does not vanish near $b\Omega$ (see \eqref{T-rep}). It is
easy to see that the conclusion of Proposition \ref{P:antiderivative} holds for $T$ if and
only if it holds for $T_{1}+L$ (however,  the $\mathcal{H}_{j}^{k}$'s in (i) change and so
do the constants $C_{j,k}$ in (ii), see also the short discussion surrounding
\eqref{E:TbTchange}). Therefore, it may be assume that $T=T_{1}+L$, with $T_{1}$ and $L$
as above. Set $\mathcal{N} := JT_{1}$. Because $T_{1}$ is complex transversal,
$\mathcal{N}$ is transversal to $b\Omega$, and so the general set-up of
\S\ref{SS:setup} applies.

The proof of Proposition \ref{P:antiderivative} will be achieved in two steps. The first
step consists in replacing $\mathcal{N}$ in Lemma \ref{L:FTC} by $i\overline{T}$, adding
the necessary correction, and then iterating the result. The key for the proof of
Proposition \ref{P:antiderivative} is that when applied to a holomorphic function, the
correction term is benign, as a result of the Cauchy--Riemann equations, see
\eqref{holbenign} below.

\begin{lemma}\label{L:zetahformula}
Let $h\in C^{\infty}(\Omega)$ for some $k\in\mathbb{N}$. Let $\zeta\in
C^{\infty}_{0}(\overline{\Omega}\cap U)$ a non-negative function which
is identically $1$ in a neighborhood of $b\Omega$ contained in $U$. Then
  \begin{align}\label{E:zetahformula}
    \left(\zeta h\right)(x)=i^{k}\left(\mathfrak{A}\circ\overline{T}\right)^{k}[\zeta h]+
\sum_{j=0}^{k-1}i^{j}\left(\mathfrak{A}\circ\overline{T}
\right)^{j}\circ\mathfrak{A}\circ(\mathcal{N}-i\overline{T})[\zeta h].
  \end{align}
\end{lemma}

\begin{proof}
  The proof is done by induction on $k$. Suppose first that $k=1$. Then Lemma \ref{L:FTC}
yields
  \begin{align}\label{E:ATk=1}
    \zeta h=\mathfrak{A}[\mathcal{N}(\zeta h)]&=\mathfrak{A}[i\overline{T}(\zeta h)]+
\mathfrak{A}[(\mathcal{N}-i\overline{T})(\zeta h)]\notag\\
    &=i\left(\mathfrak{A}\circ\overline{T}\right)[\zeta h]+\mathfrak{A}\circ
(\mathcal{N}-i\overline{T})[\zeta h] \; .
  \end{align}
For the induction step suppose that
  \begin{align}\label{E:ATkstep}
    \zeta h=i^{k-1}\left(\mathfrak{A}\circ\overline{T}\right)^{k-1}[\zeta h]+
\sum_{j=0}^{k-2}i^{j}\left(\mathfrak{A}\circ\overline{T} \right)^{j}\circ\mathfrak{A}
     \circ(\mathcal{N}-i\overline{T})[\zeta h]
 \end{align}
 holds. Using identity \eqref{E:ATk=1} to replace  $\zeta h$ in the first term
of the right hand side of \eqref{E:ATkstep} gives the result.
\end{proof}

In the second step, we will write the powers
$\left(\mathfrak{A}\circ\overline{T}\right)^{\ell}$ in terms of compositions of the form
$\left(\overline{T}\right)^{m}\circ \widetilde{\mathfrak{A}}$ with good control of $\widetilde{\mathfrak{A}}$. This is accomplished in
the following lemma; its proof relies heavily on the machinery of 
\S\ref{SSspaces} and \S\ref{SScommutators}.

\begin{lemma}\label{L:ATellstuff}
Let $A\in\mathcal{A}_{0}^{1}$, and $X$ a vector field with smooth coefficients on
$\overline{\Omega}$. Then, for any $\ell\in\mathbb{N}$, there exist operators
$G_{m}^{\ell}$, $m\in\{0,\dots,\ell\}$, which belong to 
  $\sum_{\left\{\nu\leq\ell-m, |\alpha|-\nu\geq 0
\right\}}\mathcal{A}_{\alpha,\nu}^{\ell}$ 
such that
  \begin{align*}
    \left(A\circ X \right)^{\ell}=\sum_{m=0}^{\ell}X^{m}\circ G_{m}^{\ell}.
  \end{align*}
\end{lemma}
\begin{proof}
The proof is again by induction on $\ell$. For $\ell=1$, commuting $A$ by $X$ yields
  \begin{align*}
    A\circ X=X\circ A+\left[A,X\right]=:X\circ G_{1}^{1}+G^{1}_{0}.
  \end{align*} 
It follows from the definition of $A$ that $G_{1}^{1}=A\in\mathcal{A}_{0,0}^{1}$.
Furthermore, Lemma \ref{P:basiccommutator} gives that $G_{0}^{1}$ belongs to
$\mathcal{A}_{0,0}^{1} + \mathcal{A}_{1,1}^{1}$.
      
For the induction step suppose that
  \begin{align*}
    \left(A\circ X \right)^{\ell-1}=\sum_{m=0}^{\ell-1}X^{m}\circ G_{m}^{\ell-1}
  \end{align*}
 holds for some $G_{m}^{\ell-1}\in \sum_{\left\{\nu\leq\ell-1-m, |\alpha|-\nu\geq 0
\right\}}\mathcal{A}_{\alpha,\nu}^{\ell-1}$.
 Then
  \begin{align}\label{E:ATellstuff}
     \left(A\circ X \right)^{\ell}&=\sum_{m=0}^{\ell-1}A\circ X^{m+1}\circ G_{m}^{\ell-1}
     =\sum_{m=1}^{\ell}A\circ X^{m}\circ G_{m-1}^{\ell-1}\notag\\
     &=\sum_{m=1}^{\ell}\left(X^{m}\circ A\circ
G_{m-1}^{\ell-1}+\left[A,X^{m}\right]\circ 
G_{m-1}^{\ell-1} 
     \right).
  \end{align}
 Since $A\in\mathcal{A}_{0,0}^{1}$, it follows that
  \begin{align*}
  A\circ G_{m-1}^{\ell-1}\in\sum_{\left\{\nu\leq\ell-m, 
|\alpha|-\nu\geq 0 \right\}} \mathcal{A}_{\alpha,\nu}^{\ell}\,.
 \end{align*}
To deal with the second term on the right-hand side of \eqref{E:ATellstuff}, we use the
 formula
 \begin{align}\label{commexpand}
   \left[A,X^{m} \right]=
   \sum_{j=0}^{m-1}{m \choose j} \,X^{j}\circ C_{X}^{m-j}(A),
  \end{align} 
where the iterated commutators $C_{*}^{*}$ are defined just before Lemma
\ref{C:basiccommutator1}. (\eqref{commexpand} is purely algebraic and is easily proved by
induction on $m$; alternatively, see \cite{DerridjTartakoff76}, Lemma 2 or
\cite{Straube10}, formula (3.54).) What must be shown then, is that
   \begin{align*}
     C_{X}^{m-j}(A)\circ G_{m-1}^{\ell-1}
     \in \sum_{\left\{\nu\leq\ell-j, |\alpha|-\nu\geq 0
\right\}}\mathcal{A}_{\alpha,\nu}^{\ell}\; .
   \end{align*}
For that, recall first that Lemma \ref{C:basiccommutator1} says that
$C_{X}^{m-j}(A)\in\sum_{k=0}^{m-j}\mathcal{A}_{k,k}^{1}$. Furthermore, it follows from
Lemma \ref{C:basiccommutator3} that for $|\beta|\leq k$
   \begin{align*}
     D^{\beta}\circ G_{m-1}^{\ell-1}=G_{m-1}^{\ell-1}\circ D^{\beta}+\left[D^{\beta},
G_{m-1}^{\ell-1}\right]
     \in\sum_{\left\{\nu\leq\ell-m, |\alpha|-\nu\geq
0 \right\}}\mathcal{A}_{\alpha,\nu+k}^{\ell-1}\;.
   \end{align*}
  It then follows that $C_{X}^{m-j}(A)\circ G_{m-1}^{\ell-1}$ is contained in
\begin{align*}
\sum_{\left\{\nu\leq\ell-m, |\alpha|-\nu\geq
0\,,\,0\leq k\leq m-j \right\}}\mathcal{A}^{\ell}_{(k,\alpha),\nu+k} \; .
\end{align*}   
Set $\widetilde{\nu}=\nu+k$, then $\nu\leq\ell-m$ implies that $\widetilde{\nu}\leq
\ell-j$, since $k\leq m-j$. Moreover, setting  $\widetilde{\alpha}=(k,\alpha)$ yields
$|\widetilde{\alpha}|-\widetilde{\nu}=|\alpha|-\nu\geq 0$. Hence
   \begin{align*}
     C_{X}^{m-j}(A)\circ G_{m-1}^{\ell-1}\in
     \sum_{\left\{\widetilde{\nu}\leq\ell-j, |\widetilde{\alpha}|-\widetilde{\nu}\geq 0
\right\}}\mathcal{A}^{\ell}_{\widetilde{\alpha},\widetilde{\nu}} \;,
   \end{align*}
which completes the proof.
\end{proof}

We are now ready to prove Proposition \ref{P:antiderivative}. This amounts to combining
Lemmas \ref{L:zetahformula} and \ref{L:ATellstuff},  to obtain a representation (i) with
$\mathcal{H}^{k}_{j} \in C^{\infty}(\Omega)$; this works for $h \in C^{\infty}(\Omega)$.
The final step then consists in obtaining the required estimates (ii) and membership
in $H^{k}(\Omega)$ when $h$ is holomorphic. 

\begin{proof}[Proof of Proposition \ref{P:antiderivative}]
For $h$ and $\zeta$ given as  in Lemma \ref{L:zetahformula}, use Lemmas
\ref{L:zetahformula} and \ref{L:ATellstuff} (with $\mathfrak{A}$ and $\overline{T}$ in
place of $A$ and $X$, respectively) in
\eqref{E:zetahformula} to obtain that 
$$\zeta h=\sum_{m=0}^{k}\overline{T}^{m}\mathcal{H}_{m}^{k},$$ 
where
\begin{align}\label{H-km1}
  \mathcal{H}_{k}^{k} := G_{k}^{k}[i^{k}\zeta h],
\end{align}
and
\begin{align}\label{H-km2}
  \mathcal{H}_{m}^{k} := \left(i^{k}G_{m}^{k}+\sum_{j=m}^{k-1}i^{j}G_{m}^{j}
\circ\mathfrak{A}\circ(\mathcal{N}-i\overline{T}) \right)[\zeta h]\;,\; 0 \leq m \leq k-1.
\end{align}
Because $h \in C^{\infty}(\Omega)$, so are the $\mathcal{H}^{k}_{m}$, $0\leq m \leq k$.
This establishes the representation (i) in Proposition \ref{P:antiderivative}, except for
membership in $H^{k}(\Omega)$.

We now prove the estimates (ii). By Lemmas \ref{L:ATellstuff}  and \ref{C:Aellestimate},
it follows that
\begin{align*}
  \left\|\mathcal{H}_{k}^{k}\right\|\lesssim\left\|t_{x}^{k}\cdot \zeta h\right\|\lesssim\|h\|_{-k},
\end{align*}
where the second step follows from \eqref{E:eqnorms2}; we use here that $h$ in
Proposition \ref{P:antiderivative} is holomorphic. By analogous reasoning, we have
\begin{align}\label{G-mk}
\|G_{m}^{k}[\zeta h]\| \lesssim \sum_{|\beta| \leq
|\alpha|}\|t_{x}^{k+|\alpha|}D^{\beta}h\| \lesssim \|D^{\beta}h\|_{-k-|\alpha|} \lesssim
\|h\|_{-k},
\end{align}
where $\alpha$ is such that $G_{m}^{k} \in \mathcal{A}^{k}_{\alpha,\nu}$ for some $\nu
\leq |\alpha|$. Note that $D^{\beta}h$ is also holomorphic, so
that \eqref{E:eqnorms2} applies. The last inequality holds because $|\beta| \leq
|\alpha|$.

To obtain the claimed estimates for the remaining terms in $\mathcal{H}_{m}^{k}$, first
note that because $G^{j}_{m} \in\sum_{\nu\leq j-m,|\alpha|-\nu\geq
0}\mathcal{A}^{j}_{\alpha,\nu}$, $G^{j}_{m}\circ\mathfrak{A}$ is a sum of terms of the
form
\begin{align*}
A^{1}_{\alpha_{1}}\circ A^{2}_{\alpha_{2}}\circ\cdots\circ A^{1}_{\alpha_{j}}\circ
D^{\gamma}\circ\mathfrak{A} = A^{1}_{\alpha_{1}}\circ\cdots\circ
A^{1}_{\alpha_{j}}\circ\mathfrak{A}\circ D^{\gamma} + A^{1}_{\alpha_{1}}\circ\cdots\circ
A^{1}_{\alpha_{j}}\circ\left[D^{\gamma},\mathfrak{A}\right],
\end{align*}
where $|\gamma| \leq |\alpha| \leq \nu \leq j-m \leq j \leq k-1$. Because
$\left[D^{\gamma},\mathfrak{A}\right] \in \mathcal{A}^{1}_{0,|\gamma|-1} +
\mathcal{A}^{1}_{1,|\gamma|}$ (Lemma \ref{L:basiccommutator2}), it follows, in view of
Lemma \ref{C:A1estimate}, that 
\begin{align}\label{G-jm}
\left\|G^{j}_{m}\circ\mathfrak{A}\left(\mathcal{N} - i\overline{T}\right)[\zeta h]\right\|
\; \lesssim \;  \left\|\left(\mathcal{N} - i\overline{T}\right)[\zeta h]\right\|_{k-1}.
\end{align}
The Cauchy--Riemann equations for $h$ yield
\begin{align}\label{E:YNpass} 
  \mathcal{N}h=iT_{1}h=i\overline{T_{1}}h = i\overline{T}h
\end{align}
since $T_{1}=\overline{T_{1}}$ and $\overline{L}h=0$. Consequently,
\begin{align}\label{holbenign}
  (\mathcal{N}-i\overline{T})[\zeta
h]=\left((\mathcal{N}-i\overline{T})[\zeta]\right)h.
\end{align}
Since $\left((\mathcal{N}-i\overline{T})[\zeta]\right)$ has compact support in
$\Omega\cap U$ that does not depend on $h$, it follows (for example from
\eqref{E:eqnorms2} by using that factors of $r$ are bounded away from zero on this
support, so that introducing them will at most `increase' the norm) that there exists a
constant $C_{k}$ such that
\begin{align}\label{holweakstrong}
  \left\|  (\mathcal{N}-i\overline{T})[\zeta h]\right\|_{k-1}
= \left\| 
\left( (\mathcal{N}-i\overline{T})[\zeta]\right) h\right\|_{k-1}
\leq C_{k}\|h\|_{-k} \; .
\end{align}
Combining \eqref{G-mk}, \eqref{G-jm}, and \eqref{holweakstrong} gives that
\begin{align*}
  \left\|\mathcal{H}_{m}^{k} \right\|\lesssim \|h\|_{-k}
\end{align*}
for all $m\in\{0,\dots,k-1\}$. This concludes the proof of (ii).

It remains to see that $\mathcal{H}^{k}_{m} \in H^{k}(\Omega)$. First note that because
$(\mathcal{N}-i\overline{T})[\zeta h] \in C^{\infty}_{0}(\Omega)$, it
follows easily from the form of the operators $G^{j}_{m}$ that $G^{j}_{m}\circ
\mathfrak{A}\circ (\mathcal{N}-i\overline{T})[\zeta h] \in C^{\infty}(\overline{\Omega})
\subset H^{k}(\Omega)$, $j \leq m \leq (k-1)$. The remaining contributions in
\eqref{H-km1} and \eqref{H-km2} that need to be checked are of the form $G^{k}_{m}[\zeta
h]$, $0 \leq m \leq k$. That is, we need to show that $D^{\beta}G^{k}_{m}[\zeta h] \in
L^{2}(\Omega)$ for $|\beta| \leq k$. But $D^{\beta}G^{k}_{m} = G^{k}_{m}D^{\beta} -
\left[G^{K}_{m}, D^{\beta}\right]$. The argument is now analogous to the
discussion above. For example, in view of Lemmas
\ref{L:ATellstuff}, \ref{C:basiccommutator3}, and \ref{C:Aellestimate}, the commutator
$\left[G^{k}_{m}, D^{\beta}\right]$ is a sum of terms in
$\mathcal{S}^{k+|\alpha|}_{\nu+|\beta|-1} + \mathcal{S}^{k+|\alpha|+1}_{\nu+|\beta|}
\subseteq \mathcal{S}^{k+|\alpha|}_{\nu+|\beta|}$, with $\nu\leq k-m$ and $|\alpha|\geq
\nu$. Therefore, arguing as in \eqref{G-mk}, we see that the contribution of each of these
terms to $\left\|\left[G^{k}_{m}, D^{\beta}\right](\zeta h)\right\|$ is dominated by
$\sum_{|\gamma|\leq \nu+|\beta|}\|D^{\gamma}h\|_{-k-|\alpha|}$. Because $|\gamma|\leq
\nu+|\beta|\leq |\alpha|+|\beta|\leq |\alpha|+k$, all these terms are indeed dominated
by $\|h\|$.

The argument for $G^{k}_{m}D^{\beta}$ is similar. This concludes the proof of Proposition
\ref{P:antiderivative}.
\end{proof}


\section{Proof of Theorem \ref{T:holconjsmoothing}}\label{S:5}

\begin{proof}[Proof of Theorem \ref{T:holconjsmoothing}]
As in the proof of Theorem \ref{T:main}, we invoke duality via Proposition
\ref{P:pieceofduality} and Remark \ref{dualityRemark}:
it suffices to show that
    \begin{align}
      \left|\int_{\Omega}f\overline{g} \right|\leq C \|f\|\cdot\|g\|_{-k_{1}}
    \end{align}
    for $f\in \overline{A^{0}}$ and $g\in A^{0}(\Omega)$.  We now use that for
holomorphic functions, membership in $A^{-m}(\Omega)$ for some $m$ is equivalent to having
a blow up rate near the boundary of at most a power of $1/d_{b\Omega}(z)$, where
$d_{b\Omega}(z)$ is the boundary distance function. We have the estimates

    \begin{align}\label{E:Bell}
      C_{m}^{1}\|h\|_{-m -2n-2}\leq\sup_{z\in\Omega}|h(z)|\cdot
d_{b\Omega}(z)^{m+2n}
\leq C_{m}^{2}\|h\|_{-m};
    \end{align}
    in both estimates, if the right-hand side is finite, then so is the left hand side 
(and the estimate holds; that is, the estimates are genuine estimates as opposed to \emph{a
priori} estimates). The inequalities \eqref{E:Bell} are essentially Lemma 2 in
\cite{Bell82b}, except that the norm in the left most term there is the
$(-m-4n)$-norm. The stronger version given here is in \cite{Straube84}, Theorems 1.1
and 1.3; see in particular the proof of the implication (iii) $\Rightarrow$ (iv) in
Theorem 1.1. 

Applying the second inequality in \eqref{E:Bell} to $\overline{f}$ (for $m=0$) and to $g$ (for $m=k_1$)
yields

\begin{equation*}
\sup_{z\in\Omega}|f(z)|\cdot d_{b\Omega}(z)^{2n}\leq C_{0}^{2}\,\|f\| \quad\text{ and }\quad
\sup_{z\in\Omega}|g(z)|\cdot d_{b\Omega}(z)^{k_1+2n}\leq C_{k_1}^{2}\,\|g\|_{-k_1}.
\end{equation*}
Multiply these inequalities, then apply the first inequality in \eqref{E:Bell} to
$\overline{f}g$ (which is holomorphic). The conclusion is
that $\overline{f}g$, hence $f\overline{g}$, belongs to $H^{-k_{1}-4n-2}(\Omega)$, and

\begin{align}\label{E:barfgestimate}
  \left\|\overline{f} g\right\|_{-k_{1}-4n-2}\leq C_{k_{1}}\|f\|\cdot\|g\|_{-k_{1}}.
\end{align}  
  
On the other hand, the following estimate also holds:

\begin{align}\label{E:fbargestimate}
  \left|\int_{\Omega}f\overline{g} \right|=\left|\int_{\Omega}\overline{f}g 
\right|\leq\widetilde{C}_{k_{1}}\left\|\overline{f}g
\right\|_{-k_{1}-4n-2}\|1\|_{k_{1}+4n+2}.
\end{align}
The inequality in \eqref{E:fbargestimate} is Proposition 1.9 in \cite{Straube84} which 
gives this estimate for the pairing of a harmonic function in some $H^{-m}(\Omega)$ with
an arbitrary function in $C^{\infty}(\overline{\Omega})$. Note that since both $f$ and $g$
are in $L^{2}(\Omega)$, $\overline{f}g$ is integrable, and the integral denoted by
$\widetilde{\int}$ in \cite{Straube84} coincides with the ordinary integral over $\Omega$.
Combining \eqref{E:barfgestimate} and \eqref{E:fbargestimate} completes the proof of
Theorem \ref{T:holconjsmoothing}.
\end{proof}

\begin{remark}\label{duality}
When Condition R holds, one can extend $B$ by duality to a projection $\widetilde{B}$ from
the dual $\left(C^{\infty}(\overline{\Omega})\right)^{*}$ into
$\cup_{k=1}^{\infty}A^{-k}(\Omega)$. This was observed in \cite{Kerzman72}, in a note
added in proof (where the idea is attributed to Nirenberg and Tr\`{e}ves). The conclusion
of Theorem \ref{T:holconjsmoothing} remains true for this extended projection: when
$f$ is in $\overline{\cup_{k=1}^{\infty}A^{-k}(\Omega)}$, then $\widetilde{B}f$ is smooth
up to the boundary (see \cite{Straube84}, Theorem 3.4, for the `canonical' inclusion
$\overline{\cup_{k=1}^{\infty}A^{-k}(\Omega)} \hookrightarrow
\left(C^{\infty}(\overline{\Omega})\right)^{*}$). A similar discussion applies when
\eqref{main1} holds. 
\end{remark}
 
\vskip .5cm

\bibliographystyle{plain}
\bibliography{HMS}  

\end{document}